\documentclass[12pt,reqno]{amsart}

\usepackage{latexsym}
\usepackage{amscd, amsfonts, mathrsfs, amsmath, amssymb, amsthm, mathtools}
\usepackage{tikz-cd}
\usepackage[new]{old-arrows}

\usepackage[color,arrow,matrix,curve]{xy}

\usepackage[colorlinks=true,
linkcolor=black,
citecolor=black,
filecolor=black,
urlcolor=black,
bookmarks=true,
bookmarksopen=true,
bookmarksopenlevel=3,
plainpages=false,
pdfpagelabels=true]{hyperref}

\usepackage{cleveref}
\usepackage{wasysym}
\usepackage[OT2,T1]{fontenc}

\newtheorem{thm}{Theorem}[section] 
\crefname{thm}{theorem}{theorems}

\crefname{algo}{algorithm}{algorithms}

\crefname{claim}{claim}{claims}

\crefname{conj}{conjecture}{conjectures}
\newtheorem{cor}[thm]{Corollary}
\crefname{cor}{corollary}{corollaries}
\newtheorem{defn}[thm]{Definition}
\crefname{defn}{definition}{definitions}

\crefname{exer}{exercise}{exercises}
\newtheorem{Ex}[thm]{Example}
\crefname{Ex}{example}{examples}
\newtheorem{lemma}[thm]{Lemma}
\crefname{lemma}{lemma}{lemmas}
\newtheorem{prop}[thm]{Proposition}
\crefname{prop}{proposition}{propositions}

\crefname{prob}{problem}{problems}
\newtheorem{rem}[thm]{Remark}
\crefname{rem}{remark}{remarks}

\crefname{notation}{notation}{notation}

\crefname{ques}{question}{questions}
\newtheorem{conjecture}[thm]{Conjecture}

\def\ra{\rightarrow}

\def\Stab[#1]{\text{Stab}_G \left( #1 \right)}
\def\ZZ{\mathbb{Z}}

\def\ed[#1][#2]{\text{ed}_{#1} \left( #2 \right)}
\def\cd[#1]{\text{cd}_k \left( #1 \right)}
\def\cd[#1][#2]{\text{cd}_{#1} \left( #2\right)}
\def\edp[#1]{\text{ed}_k\left( #1 ; p \right)}

\def\O{\mathcal{O}}

\def\QQ{\mathbb{Q}}
\def\ZZ{\mathbb{Z}}

\def\HGm[#1]{H^2\left( #1_{ét}, \Gm \right)}

\def\Brone[#1][#2]{\text{Br}_1 \left( #1\right)_{#2}}
\def\Brone[#1]{\text{Br}_1 \left( #1\right)}

\def\tor[#1][#2]{\prescript{}{#1}{#2}}

\def\Gm{\mathbb{G}_m}

\def\xcoor[#1]{\mathbf{x}\left( #1 \right)}
\def\ycoor[#1]{\mathbf{y}\left( #1 \right)}
\renewcommand{\phi}{\varphi}

\DeclareMathOperator{\Cor}{cor}
\DeclareMathOperator{\res}{res}

\DeclareMathOperator{\Br}{Br}
\DeclareMathOperator{\Spec}{Spec}

\DeclareMathOperator{\cohd}{cd}
\DeclareMathOperator{\ind}{ind}
\DeclareMathOperator{\syl}{symb}

\DeclareMathOperator{\period}{per}

\DeclareMathOperator{\dd}{dd}

\DeclareSymbolFont{cyrletters}{OT2}{wncyr}{m}{n}
\DeclareMathSymbol{\Sha}{\mathalpha}{cyrletters}{"58}

\allowdisplaybreaks

\numberwithin{equation}{section}

\usepackage[]{geometry}

\def\Ker{{\rm Ker}}
\def\Supp{{\rm supp}}
\def\div{{\rm div}}
\def\X{\mathfrak{X}}

\def\Z{\mathbb{Z}}
\def\Q{\mathbb{Q}}
\def\O{\mathcal{O}}

\title{Period-index in top cohomology over semiglobal fields}

\author[Dijols]{Sarah Dijols} 
\address{(Dijols) University of British Columbia, Vancouver, BC V6T 1Z4, Canada}
\email{sarah.dijols@math.ubc.ca}

\author[Parimala]{Raman Parimala}
\address{(Parimala) Emory University, Atlanta, GA 30322, U.S.A.}
\email{parimala.raman@emory.edu}

\author[Ramdorai]{Sujatha Ramdorai}
\address{(Ramdorai) University of British Columbia, Vancouver, BC V6T 1Z4, Canada}
\email{sujatha@math.ubc.ca}

\author[Ure]{Charlotte Ure}
\address{(Ure) Illinois State University, Normal, Illinois 61790, U.S.A.}
\email{cure@ilstu.edu}



\date{\today}

\begin{document}

\begin{abstract}
    We prove a common slot lemma for symbols in top cohomology classes over semiglobal fields. Furthermore, we prove that period and index agree for general top cohomology classes over such fields. We discuss applications to quadratic forms and related open problems. 
\end{abstract}

\maketitle

\section{Introduction}\label{sec:introduction} 

Let $F$ be a field and fix a separable closure $\overline{F}$ of $F$. The Galois group $\mathrm{Gal}\left( \overline{F}/F\right)$ will be denoted by $G_F$. Let $\mu_m$ be the group of $m$-th roots of unity in $\overline{F}$, considered as a module over $G_F$, and denote by $\mu_m^{\otimes i}$ its $i$-fold tensor product.  The Galois cohomology groups $H^i\left(G_F,\mu_m^{\otimes i}\right)$ will be denoted $H^i\left(F,\mu_m^{\otimes i}\right)$ throughout. For any field extension $E$ of $F$, the profinite group $G_E$ is a subgroup of $G_F$ and so there are natural restriction maps 
\begin{equation} \label{eq:restriction} \res: H^i \left( F, \mu_m^{\otimes i}\right) \ra H^i\left( E, \mu_m^{\otimes i}\right).\end{equation}
that take a cohomology class $\xi$ to its restriction $\xi_E$. If $\xi_E =0$, we call $E$ a splitting field of $\xi$. 

We call an element $\xi \in H^i\left( F,\mu_m^{\otimes i}\right)$ a \emph{symbol} if it is in the image of the $i$-fold cup product map 
$$H^1 \left( F, \mu_m\right) \times \cdots \times H^1\left( F, \mu_m\right) \rightarrow H^i \left( F, \mu_m^{\otimes i}\right). $$
If $\mu_m \subset F$, then $\mu_m^{\otimes i} \cong \mu_m$ as $G_F$-modules, so that there is a notion of symbols in $H^i \left( F, \mu_m\right)$. 
In the case that $i=m=2$, a symbol in $H^2\left( F, \mu_2^{\otimes 2}\right)$ determines a quaternion algebra over $F$ and vice versa. Tate's common slot lemma states that if $F$ is a number field, then any finite set of symbols in $H^2\left( F, \mu_m^{\otimes 2}\right)$ can be split by a single degree $m$ extension of $F$ \cite[Theorem 4.4 and the remarks following the theorem]{Tate1976}. In particular, any finite set of quaternion algebras over a number field may be split by a single quadratic extension of $F$. A semiglobal field $F$ is a  function field in one variable over a complete discretely valued field $K$. We prove an analogue of Tate's common slot lemma in the setting of semiglobal fields in \Cref{commonslot}. 

Understanding the splitting fields of general cohomology classes, as in Tate's common slot lemma, has been an interesting problem in the study of  Galois cohomology. This motivates the notion of the index of $\xi \in H^i(F,\mu_m)$ as the greatest common divisor of the degrees of field extensions $E$ of $F$ so that $\xi$ splits over $E$.
For an element $\xi$ in $H^i(F,\mu_m)$, we can also define the period of $\xi$ to be its order, which is denoted by $\period(\xi)$.  It is known that the period divides the index and that they have the same prime factors (\cite{Pie82}, Proposition 14.4(b)(ii)). The period-index question asks for the smallest number $r$ so that 
$$\ind(\xi)|\period(\xi)^{r}$$
for all $\xi \in H^i(F, \mu_m)$. There has been much progress toward answering this problem over various fields. In the case that $F$ is a local or global field,  period and index agree in the Brauer group ($i=2$) by the  Albert--Brauer--Hasse--Noether Theorem \cite{Brauer-Hasse-Noether1932}. 
Recent important results have been obtained in the case of degree 3 cohomology, where \cite{PS98} showed that period and index coincide in $H^3\left(F, \mu_m\right)$ if $F$ is the function field of a $p$-adic curve or the function field of a surface over a finite field. Furthermore, this coincidence has been settled for function fields of curves over imaginary number fields in \cite{Sur20}. In \cite{HHK22}, the authors obtain bounds on the index in terms of the period for $F$ a semiglobal field. 

Throughout this paper, $F$ is a semiglobal field, i.e. a function field in one variable over a complete discretely value field $K$ 
with residue field $k$. Let $m$ be a number coprime to the characteristic of $k$, and $n$ be the largest number such that $H^n(k, \mu_m)$ does not vanish. We investigate the cohomology groups $H^{n+2}(F, \mu_m)$. Since $H^i(F,\mu_m)$ vanishes for any $i>n+2$, we call $H^{n+2}(F, \mu_m)$ the \emph{top cohomology}. For classes $\xi \in H^{n+2}(F, \mu_m)$, Gosavi proved in \cite[Corollary 3.3]{Gosavi22} that $\ind(\xi) | \period(\xi)^2$ for $m$ an odd prime and $\ind(\xi) |\period(\xi)^3$ for $m =2$. In \Cref{thm:maintheorem}, we prove that period and index agree in top cohomology for semiglobal fields for any $m$ coprime to the characteristic of $k$. In fact, for a class $\xi \in H^{n+2}(F, \mu_m)$, we construct a field extension $L$ of $F$ of degree $m$ that is a splitting field of $\xi$.  \\  

This paper is organized as follows. In \Cref{sec:prelim}, we give necessary background on cohomological dimension, the period-index problem, semiglobal fields, and the local-global principle. Let $K$ be a complete discretely value field with residue field $k$ of characteristic coprime to $m$ and suppose that $k$ has cohomological dimension $n$. In \Cref{sec:complete}, we explore the top cohomology groups $H^{n+1}(K, \mu_m)$ of the complete discretely value field $K$. In particular, we prove \Cref{lemma:decompose_xi_over_K}, which provides a decomposition of elements in $H^{i}(K, \mu_m)$ that is crucial to our proofs. Using this decomposition, we show that period and index agree in the top cohomology $H^{n+1}(K, \mu_m)$ groups in \Cref{mun}. In \Cref{main}, we state and prove our two main theorems. First, we prove an analogue of Tate's common slot lemma over semiglobal fields in \Cref{commonslot}. We show that for any finite set $\mathcal{S}$ of symbols in $H^{n+2}\left( F, \mu_m^{\otimes n+2}\right)$, there is a field extension $L$ of $F$ of degree $m$ so that all $s \in \mathcal{S}$ are split over $L$. The proof relies on the local-global principle discussed in \Cref{HHK}. More precisely, we show that for any discrete valuation $\omega$ on $L$ with restriction $\nu$ to $F$, the restriction $s_{F_\nu}$ splits over $s_{L_\omega}$ for all $s \in \mathcal{S}$. Finally, we show in \Cref{thm:maintheorem} that period and index agree in top cohomology of $H^{n+2}\left(F, \mu_m\right)$. We conclude the paper by discussing some open questions on top cohomology and connections of period-index questions with the $u$-invariant of quadratic forms in \Cref{sec:open}. \\

\noindent\textbf{Acknowledgement.} 
First, we thank Suresh Venapally for bringing to our attention the questions on top cohomology and for continued support during the project. We also thank Danny Krashen for helpful comments and suggestions. 

This project was started during the Women in Numbers (WIN) workshop at the Banff International Research Station in 2023. We thank the workshop organizers Shabnam Akhtari, Alina Bucur, Jennifer Park, Renate Scheidler, and the staff at BIRS for their support. We also thank the following for supporting the workshop, childcare, and travel: NSF DMS-2012061, the Number Theory Foundation, the Clay Mathematics Institute, the Journal of Number Theory, the Pacific Institute for Mathematical Sciences, and the AMS-Simons Travel Grant program. Sujatha Ramdorai is supported by the NSERC Discovery Grant 2019-03987.

\section{Preliminaries}\label{sec:prelim} 

\subsection{Cohomological dimension} \label{subsec:cd}

For a profinite group $\Gamma$ and a $\Gamma$-module $A$, we denote by $H^i(\Gamma,A)$ the $i$-th profinite cohomology group. If $\Gamma_F$ is the absolute Galois group of a field $F$, we also call $A$ an $F$-module and write 
$H^i(F,A) = H^i(\Gamma_F,A)$ (see for instance \cite[Chapter I \S 2.2 and Chapter II]{Serre1973}). The cohomological dimension is an important invariant associated with a field (for more details, we refer the reader to \cite[Chapter I \S 3 and Chapter II \S 4]{Serre1973}).
\begin{defn}
For a prime $\ell$, the \emph{$\ell$-cohomological dimension of F, $\cohd_\ell(F)$,} is the smallest integer $R$ for which the $\ell$-primary component of $H^{i+R}(F,A)$ is zero for any integer $i \geq 1$ and for every torsion $F$-module $A$. If no such $R$ exists, we say that $\cohd_\ell(F) =\infty$.
Analogously, the \emph{cohomological dimension $\cohd(F)$ of $F$ }is the smallest integer $R$ so that $H^{i+R}(F,A) =0$ for any $i \geq 1$ and for any torsion $F$-module $A$. 
If no such $R$ exists, we say that $\cohd(F) = \infty$. 
\end{defn}

An algebraically closed field has cohomological dimension $0$. Finite fields and the Laurent series field $k((t))$ with $k$ algebraically closed of characteristic $0$ have cohomological dimension $1$. More generally, the cohomological dimension is invariant under finite field extensions (see \Cref{lemma:cd-finite-ext} below) and increases as the transcendental degree increases (see \Cref{transcendance} below).

\begin{lemma}[{\cite[Proposition 10, Chapter II, \S 4.1]{Serre1973}}]\label{lemma:cd-finite-ext}
    Let $L$ be an algebraic extension of the field $F$ and $\ell$ a prime number. Assume $[L:F] < \infty$ and $\cohd_\ell(F) < \infty$, then $\cohd_\ell(L) = \cohd_\ell(F)$. In particular, if  $[L:F] < \infty$ and $\cohd(F)<\infty$, then $\cohd(L) = \cohd(F)$. 
    \end{lemma}

\begin{lemma}[{\cite[Proposition 11, Chapter II, \S 4.2]{Serre1973}}]\label{transcendance}
    Let $L$ be an extension of $K$, then 
    $$\cohd_\ell(L) \leq \cohd_\ell(K) + \mathrm{trdeg}_K (L)$$
    for any prime $\ell$. There is equality if $L$ is finitely generated over $K$, $\cohd_\ell(K) <\infty$, and $\ell$ is coprime to the characteristic of $K$.  
\end{lemma}

\begin{lemma}\label{cdK=cdk+1} 
    Let $K$ be a complete discretely value field with residue field $k$. Suppose that the characteristic of $k$ is different from $\ell$. If $\cohd_\ell(k) \leq n-1$, then $\cohd_\ell(K)\leq n$. 
\end{lemma} 

\begin{proof} Since the char$(k) \neq \ell$ and $K$ is complete, 
 there is a residue map from $H^n(K, \mu_\ell)$ to 
$H^{n-1}(k, \Z/\ell\Z)$ with kernel isomorphic to $H^n(k, \mu_\ell)$   
(cf. \cite[Proposition 12, Chapter II, \S 4.3]{Serre1973}). This proves the statement. 
\end{proof}

Let $F=K(C)$ be the function field of an algebraic curve over a complete discretely valued field $K$ with residue field $k$. Then $F$ is a semiglobal field in the sense discussed below in \Cref{semiglobal}. If $\ell$ is a prime number which is coprime to the characteristic of the residue field $k$, then it is also coprime to the characteristic of $K$ and $F$. In particular, we may apply \Cref{cdK=cdk+1} to $F$ and $K$ in this setting to deduce that if $\cohd_\ell(k) = n$, then $\cohd_\ell(K) = n+1$ and $\cohd_\ell(F) = n+2$. Furthermore, if $\cohd(k) = n$, then $\cohd_\ell(K) \leq n+1$ and $\cohd_\ell(F) \leq n+2$ for any prime $\ell$ coprime to the residue characteristic.\\

For a field $F$ with $\cohd(F) =n$ and $A$ a torsion $G$-module, the \emph{top cohomology with coefficients in $A$} is the cohomology group $H^n\left( F,A\right)$. Throughout this paper, $A$ will often be the Galois module arising from roots of unity.

\subsection{Period-index problem and Voevodsky's Theorem} \label{Voevodsky}

Let $F$ be a field and let $m$ be any number coprime to the characteristic of $F$. Fix a separable closure $\overline{F}$  of $F$ and denote by $\Gamma_F$ its absolute Galois group. Consider the Kummer exact sequence of $\Gamma_F$ -modules 
$$\xymatrix{ 1 \ar[r] & \mu_m \ar[r] & \overline{F}^\times \ar[r]^{a \mapsto a^m} & \overline{F}^\times \ar[r] & 1}. $$
The induced long exact sequence of cohomology in conjunction with Hilbert 90 gives an isomorphism \begin{equation} \label{eq:Kummer} H^1\left( F, \mu_m\right) \xrightarrow[]{\sim} F^\times/(F^\times)^m. \end{equation}
There is an $i$-fold cup-product from the product of $i$ copies of $H^1(F,\mu_m)$ to $H^i(F,\mu_m)$ for any $i>1$
\begin{equation} \label{eq:cup-product} 
    \cup: H^1\left( F, \mu_m \right) \times \cdots \times H^1\left( F, \mu_m\right) \longrightarrow H^i \left( F, \mu_m \right). 
\end{equation}
For elements $a_1, \ldots, a_i \in F^\times/(F^\times)^m,$ we denote by $(a_1, \ldots, a_i) \in H^i\left(F, \mu_m\right)$ the image under the composition of the isomorphism in (\ref{eq:Kummer}) and the cup product in (\ref{eq:cup-product}). Such an element in $H^i\left(F, \mu_m\right)$ is called a \emph{symbol}. Abusing notation, we will consider $a_1, \ldots, a_i \in F^\times$ throughout. The cup-product map in (\ref{eq:cup-product}) is alternating and we have the following relations for any $a,b \in F^\times$ and $n \in \mathbb{Z}$
\begin{equation}\label{eq:relations-for-symbols}
    \begin{aligned}
        -(a,b) &= (b,a) \\
        n(a,b) &= (a,b^n) = (a^n,b) \\
        (a,a) &= (-1,a) 
    \end{aligned}
\end{equation}
(see e.g. \cite[Chapter 3.4 and Chapter 4.6]{Gille-Szamuely2017}). 

Let $K_i^M(F)$ denote the $i$-th Milnor $K$-group of $F$ (see H. Bass’s Algebraic K-theory). 
Let $\ell$ be a prime not equal to char$(F)$. Suppose that $F$ contains the $\ell$-th roots of unity $\mu_\ell$. 
By a celebrated result of Voevodsky \cite{Voevodsky2003}, the norm-residue map determines an isomorphism between Milnor's $K$-groups and cohomology
$$K_i^M(F) \otimes \ZZ/\ell\ZZ  \xrightarrow[]{\sim} H^i\left( F, \mu_\ell\right).$$
In particular, any element $\xi$ in $H^i\left( F, \mu_\ell\right)$ may be written
 as a sum of symbols $$\xi = \sum_{j=1}^r \left( a_{j,1} , \ldots, a_{j,i} \right), $$
with $ a_{j,1} , \ldots, a_{j,i} \in F^\times$ for $1 \leq j \leq r$.
Related to this presentation, the {\it symbol length} of the $i$-th cohomology group 
$H^i\left(F, \mu_\ell\right)$, denoted $\syl_i^\ell(F)$, is defined as the 
minimal number $r$ such that any element
$\xi\in H^i(F,\mu_\ell)$ is expressible as a sum of at most $r$ symbols.  
Investigating the symbol lengths $\syl_i^\ell(F)$ for fields $F$ is of interest.
 For more background and open problems related to the symbol lengths of fields, 
 we refer the interested reader to  \cite{Auel-Brussel-Garibaldi-Vishne2011}. 
 We note that by definition, the symbol length $\syl_1^\ell(F) = 1 $   
 for any field $F$. For $n=2$, the second cohomology group 
 $H^2(F,\mu_\ell)$ is isomorphic to the $\ell$-torsion of the Brauer group
 $\Br(F)$. For a $p$-adic local field $F$, the Brauer group $\Br(F) \cong \QQ/\ZZ$ \cite[Chapter XII]{Serre1997} and if $\mu_\ell \subset F$, then $\syl_2^\ell(F) = 1$. Using this together with the 
 Albert--Brauer--Hasse--Noether Theorem \cite{Brauer-Hasse-Noether1932}, 
 we deduce that $\syl_2^\ell(F) = 1$ for any totally imaginary global field $F$. 
 In either of the two previous cases $\cohd(F) = 2$ and so in summary 
$$\syl_n^\ell (F) = \begin{cases} 1 & n =1,2 \\ 0 & n >2\end{cases}
$$ for a local or a totally imaginary global field $F$ containing $\mu_\ell$. 
Now, let $F$ be the function field of a $p$-adic curve, and let $\ell$ be a
 prime which is coprime to $p$, then the symbol length $\syl_2^\ell(F)$ is $2$ 
 \cite{Saltman1997,Saltman1998} and $\syl_3^\ell(F) = 1$ \cite{Parimala-Suresh2010}. 

\begin{rem}\label{rmk:splitting-field} We recall that a cohomology class $\xi \in H^i(F, A)$ for a 
field $F$ and a Galois module $A$ has splitting field $L$ if the restriction $\xi_L$ as defined in 
(\ref{eq:restriction}) is trivial. Any symbol $(a,b)$ in $H^n (F, \mu_m)$ is split by the degree $m$ 
extension $L= F(\sqrt[m]{a})$ since 
$$0 = m \left( \sqrt[m]{a}, b\right)_{L} \overset{(\ref{eq:relations-for-symbols})}= \left(\left( \sqrt[m]{a}\right)^m, b\right)_L = \left( a,b\right)_L,$$
where the first equality holds as $H^n(F, \mu_m)$ is $m$-torsion. 
More generally, a symbol $(a_1, \ldots, a_i) \in H^i(F, \mu_m)$ splits over the degree $m$ extension $F(\sqrt[m]{a_j})$ for any $j$. 
\end{rem}The index of a cohomology class gives a measure for splitting. 

\begin{defn}
    Let $F$ be a field and $\xi \in H^i\left(F,A\right)$. The \emph{index of $\xi$} is 
    $$\ind(\xi) = \gcd\left\{ [L:F] : \xi_L \text{ splits} \right\}.$$
\end{defn}

The period-index question discussed in the introduction explores the relationship between symbol lengths and the degrees of splitting field extensions for symbols.
In \Cref{quad} we will explain the connection between symbol lengths and the $u$-invariant of a field.

\subsection{Semiglobal fields and a local-global principle} \label{semiglobal}

One of the main tools we build on in this paper is patching in semiglobal fields. We follow \cite{HHK} in the following definitions. 

\begin{defn}[semiglobal fields]
  Let $K$ be a complete discretely valued field. A \emph{semiglobal field} over $K$ is a one-variable function field $F$ over $K$.
\end{defn}

\begin{Ex}
Examples of semiglobal fields are $F = \Q_p(x)$, $F = k((t))(x)$ for any field $k$, and any finite extension of these.
\end{Ex} 

Let $F$ be a semiglobal field, that is $F$ is the function field of a smooth
projective  curve $C$ over $K$.
Let   $\mathcal{O}_K$ be the corresponding valuation ring, $\pi \in \mathcal{O}_K$ a (fixed) 
uniformizing parameter and $k$ the residue field.
A \emph{regular proper model of $C$} is a proper $\mathcal{O}_K$-scheme $\mathfrak{X}$ with function field $F$ that is regular as a scheme. We denote its reduced closed fiber by $\mathfrak{X}_k$. By \cite[Main Theorem]{Lipman1978} such a regular proper model exists and by \cite[page 193]{Lipman1975}, we may assume that the reduced closed fiber $\mathfrak{X}_k$ is a union of regular curves with normal crossings. 
We call such a model a \emph{normal crossings model} of $F$. 
$$\xymatrix{
C \ar[d] & \mathfrak{X} \ar[d] \ar[l] \ar[r] & \mathfrak{X}_k \ar[d] \\
\Spec(K) & \Spec(\mathcal{O}_K) \ar[l] \ar[r] & \Spec(k) 
}$$
For each prime divisor $D$, we can look at the discretely valued ring $\O_{D,\mathfrak{X}}$ and get a corresponding discrete valuation $v_D$ on $F$. Such valuations are known as divisorial.

\begin{defn}
The \emph{divisorial set of places of $F$ corresponding to the model $\mathfrak{X}$} is the set 
$$\Omega_{F,\mathfrak{X}} = \left\{v_{D} \mid D \text{ prime divisor on } \mathfrak{X}\right\}.$$
We call $v_D$ with $D$ as above a divisorial valuation on $F$. The set of all divisorial valuations $v_D$, on all regular projective models $\X$ of $F$, will be denoted by $\Omega_F$. 
\end{defn}

A new approach to local-global principles for homogeneous varieties over function fields of curves defined over complete discretely valued fields was introduced by Harbater, Hartmann, and Krashen in \cite{HHK} via patching. Further, Parimala and Suresh, building on the results of Hartmann, Harbater, and Krashen, proved the following Theorem that will fundamentally be used later in our paper:
\begin{thm}[Theorem 4.3, \cite{HKP2021}] \label{HHK} Let $F$ be a semiglobal field, $\ell$ be a prime number, $\Omega_F$ the set of divisorial places of $F$, and let $r$ be any integer. Suppose that $\hbox{char}(k) \neq \ell$, and let $n+1 \geq 2$. Then
$$ \Sha_{\Omega_F}^{n+1} := \Ker\left(H^{n+1}\left(F, \mu_\ell^{\otimes_r}\right) \rightarrow \prod_{\omega \in \Omega_F}H^{n+1}\left(F_{\omega}, \mu_\ell^{\otimes_r}\right) \right)=0.$$
\end{thm}

\section{Complete discretely value field of cohomological dimension $n+1$}
\label{sec:complete}

Let $K$ be a complete discretely value field, $\O_K$ its ring of integers, $\pi$ its local uniformizer, and $k$ its residue field.  Throughout this section, we suppose that $\ell$ is a prime which is coprime to the residue characteristic and $\cohd_\ell(k) = n$. By \Cref{cdK=cdk+1}, we see that $\cohd_\ell(K)=n+1$. As a result, the expression \emph{top cohomology} of $K$ will refer to $H^{n+1}(K, \mu_\ell).$ The aim of this section is to prove that period and index agree in top cohomology in this setting. 

We first assume that $\mu_{\ell} \subset k$. We recall that that there is an exact sequence on Galois cohomology 
\begin{equation}\label{eq:ramification}
    \xymatrix{
    0 \ar[r] & H^{n+1}_{et} \left( \O_K, \mu_\ell\right) \ar[r]^r & H^{n+1} \left( K, \mu_{\ell}\right) \ar[r]^\delta & H^n \left( k, \mu_{\ell}\right), 
    }
\end{equation}
where $r$ is induced by restriction and $\delta$ is a residue map (see e.g. \cite[Chapter 6.8]{Gille-Szamuely2017}). An element $\xi$ in $H^{n+1}\left(K, \mu_\ell\right)$ is called \emph{unramified} if it lies in the kernel of $\delta$, or equivalently in the image of $r$.  Otherwise, we say that $\xi$ is \emph{ramified}. 

We show that any cohomology class splits into a \textit{sum of ramified and unramified parts}. We recall from \Cref{Voevodsky} that there is  a cup-product map from $\otimes^m H^1(K, \mu_{\ell})$ to $H^m(K, \mu_{\ell})$ and that $H^1(K, \mu_{\ell}) \cong K^{\times}\slash (K^{\times})^{\ell}$. 

\begin{lemma}\label{lemma:decompose_xi_over_K}
    Let $K$ be a discrete valued field, $\O_K$ its ring of integers, and $k$ its residue field. Suppose that $\mu_\ell \subset k$ and denote by $\pi$ the local uniformizer in $\O_K$. Any element $\xi \in H^m\left( K, \mu_\ell\right)$ decomposes as 
    $$\xi = \xi_1 + (\xi_2, \pi)$$
    with $\xi_1 \in H^m(K, \mu_\ell)$ and $\xi_2 \in H^{m-1}(K, \mu_\ell)$ unramified for any $m>0$. In the above equation $(\xi_2, \pi)$ denotes the cup product of $\xi_2$ with the symbol $(\pi) \in H^1\left( K, \mu_\ell\right)$. 
\end{lemma}

\begin{proof}
    Let $\xi \in H^{m}(K, \mu_\ell)$ and suppose first that $\xi$ is a symbol, that is 
     $\xi =  \left(a_1, \ldots, a_m\right)$ with $a_i \in K^\times$. Write $a_i = u_i \pi^{j_i}$ with $u_i$ a unit in $\O_K$ and $0 \leq j_i \leq \ell-1$. Using the relations in (\ref{eq:relations-for-symbols}), we see that
     \begin{align*}
         \xi =& \left(a_1, \ldots, a_m\right) \\
         =& \left(u_1 \pi^{j_1}, \ldots , u_m \pi^{j_m}\right)\\
         =& \left(u_1 , u_2 \pi^{j_2}, \ldots , u_m \pi^{j_m}\right) + \left( \pi^{j_1}, u_2 \pi^{j_2}, \ldots, u_m \pi^{j_m}\right)\\
         =& \left( u_1, u_2, u_3 \pi^{j_3}, \ldots, u_m \pi^{j_m}\right)
         + \left(u_1, \pi^{j_2}, u_3 \pi^{j_3}, \ldots, u_m \pi^{j_m}\right)\\
         & + \left(\pi^{j_1}, u_2, u_3 \pi^{j_3}, \ldots, u_m \pi^{j_m}\right) + \left(\pi^{j_1}, \pi^{j_2}, u_3 \pi^{j_3}, \ldots, u_m \pi^{j_m}\right) \\
         =& \left( u_1, u_2, u_3 \pi^{j_3}, \ldots, u_m^{j_m}\right) + \left( \tfrac{u_1^{j_2}}{u_2^{j_1}}, \pi, u_3 \pi^{j_3}, \ldots, u_m \pi^{j_m}\right) \\
         &+ \left( (-1)^{j_1+j_2}, \pi, u_3\pi^{j_3},\ldots, u_m\pi^{j_m}\right) \\
        \vdots\\
         =&   \left( u_1, u_2, \ldots, u_m\right) + \sum_{t}\left( v_{1}^t, v_{2}^t, \ldots, v_{m-1}^t, \pi \right)
     \end{align*}
     for some units $v_{s}^t$ in $\O_K$. Since $\mu_\ell \subset k$, any cohomology class $\xi \in H^{n+1}\left( K, \mu_\ell \right)$ may be decomposed as a sum of symbols $\xi = \sum_i ( a_{i,1}, \ldots, a_{i,m})$ by Voevodsky's Theorem (see \Cref{Voevodsky}). We decompose each summand of $\xi$ as above to write 
     $$ \xi = \sum_i (a_{i,1}, \ldots, a_{i,m}) = \sum_i ( u_{i,1}, \ldots, u_{i,m}) + \sum_{i,t} (v_{i,1}^t, \ldots, v_{i,m-1}^t,\pi) = \xi_1 + (\xi_2,\pi),$$
     which finishes the proof. 
\end{proof}
Using the description from the previous Lemma, we are now able to explicitly determine a splitting field of $\xi \in H^{n+1}\left( K, \mu_\ell\right)$. We recall from \Cref{rmk:splitting-field}, that a symbol $(a_1, \ldots, a_n) \in H^n(K,\mu_\ell)$ may be split by the extension $L = K(\sqrt[\ell]{a_{i}})$ for any $1 \leq i \leq n$.

\begin{lemma} \label{muellink}
Let $K$ be a complete discretely value field, $\pi$ its local uniformizer, and $k$ its residue field. Suppose that $\cohd_\ell(k)=n$. Assume additionally that $\mu_\ell \subset k$. Then any element $\xi \in H^{n+1}(K, \mu_{\ell})$ is split by the degree $\ell$ field extension $K(\sqrt[\ell]{\pi})$ of $K$. In particular, it has index bounded above by $\ell$.
\end{lemma} 

\begin{proof}
We note that as $\cohd_\ell(k)=n$, by \Cref{cdK=cdk+1}, we have $\cohd_\ell(K) = n+1$.
    Let $\xi \in H^{n+1}(K, \mu_\ell)$.
    By \Cref{lemma:decompose_xi_over_K}, $\xi$ can be written as $\xi = \xi_1 + (\xi_2,\pi)$ for some unramified $\xi_1 \in H^{n+1}_{et}(\O_K, \mu_{\ell})$ and $\xi_2 \in H^{n}_{et}(\O_K, \mu_\ell)$. We note that if $\xi$ is unramified, then $\xi = \xi_1$. 
    Since $\O_K$ is Henselian, we remark that $ H^{m} _{et}(\O_K,\mu_\ell) \cong H^m(k,\mu_\ell)$ for any $m$. 
Together with the fact that $\cohd_\ell(k) = n$, we deduce that $\xi_1=0$. In particular, if $\xi$ is unramified, then $\xi=\xi_1=0$ is already split over $K$ so that $\ind(\xi) = 0 \leq \ell$. In general, we write any $\xi = (\xi_2,\pi)$ with $\xi_2$ as before. This cohomology class is split by the degree $\ell$ field extension $K(\sqrt[\ell]{\pi})$ of $K$ so that $\ind(\xi)$ is again bounded above by $\ell$.  
\end{proof} 

The following lemma extends \Cref{muellink} to the case $\mu_\ell \not\subset k$. 

\begin{lemma}\label{lemma:muellnotink}
Let $K$ be a complete discretely value field, $\pi$ its local uniformizer, and $k$ its residue field. Suppose that $\cohd_\ell(k)=n$ and assume that the characteristic of $k$ is coprime to $\ell$, but possibly $\mu_\ell \not\subset k$. Then any element  $ \xi \in H^{n+1}\left(K, \mu_\ell\right)$ is split by the degree $\ell$ field extension $K(\sqrt[\ell]{\pi})$ of $K$. In particular, it has index bounded above by $\ell$. 
\end{lemma} 

\begin{proof}
We note that as $\cohd_\ell(k)=n$, by \Cref{cdK=cdk+1}, we have $\cohd_\ell(K) = n+1$.
Since $K$ does not contain $\mu_{\ell}$, let us first consider the degree $\ell-1$ extension $E= K(\rho$) where $\rho$ is a generator of $\mu_{\ell}$. We remark that the characteristic of $k$ is coprime to $\ell$, so that the cyclotomic extension $E$ over $K$ is unramified and thus $\pi$ is a local uniformizer for $E$ as well \cite[Equation (5.5) in the proof of Lemma 5.4]{cassels_1986}. By  \Cref{muellink}, after restricting to $E$, $\xi$ is split by the extension $E(\sqrt[\ell]{\pi})$. Now consider the following diagram of field extensions 
$$ 
\xymatrix{ & E (\sqrt[\ell]{\pi}) \ar@{-}[ld]_{\ell-1} \ar@{-}[rd]^\ell \\
K(\sqrt[\ell]{\pi}) \ar@{-}[rd]_{\ell} & & E \ar@{-}[ld]^{\ell-1} \\
& K }
$$
The composition 
$$ \xymatrix{
H^{n+1}\left(K(\sqrt[\ell]{\pi}), \mu_\ell\right) \ar[r]^{\res} & H^{n+1}\left(E(\sqrt[\ell]{\pi}), \mu_\ell\right) \ar[r]^{\Cor}& H^{n+1}\left(K(\sqrt[\ell]{\pi}), \mu_\ell\right)
}$$ is multiplication by $[E(\sqrt[\ell]{\pi}) : K(\sqrt[\ell]{\pi})] = \ell-1$ \cite[Corollary 1.5.7]{Neukirch}. Thus 
$$(\ell -1) \xi_{K(\sqrt[\ell]{\pi})} = \mathrm{cor} \circ \mathrm{res} \left( \xi_{K(\sqrt[\ell]{\pi})} \right)= \mathrm{cor}(0) = 0.$$ As $\ell$ and $\ell-1$ are coprime, we deduce that $\xi_{K(\sqrt[\ell]{\pi})}=0$ as well and so $K(\sqrt[\ell]{\pi})$ is a splitting field for $\xi$. 
\end{proof}
Finally, using induction on the number of prime factors, we prove the general period-index result for complete discretely value fields. 

\begin{lemma} \label{mun}
    Let $K$ be a complete discretely value field and $k$ its residue field. Suppose that the characteristic of $k$ is coprime to an integer $m>1$. Suppose that $\cohd(k) =n$. Then any element $\xi \in H^{n+1}(K, \mu_m)$ has index bounded above by $m$. 
\end{lemma}

Since $m$ is coprime to the residue characteristic and $\cohd_\ell (k) \leq \cohd(k) =n$ for any prime divisor $\ell$ of $m$, by \Cref{lemma:cd-finite-ext} and \Cref{cdK=cdk+1}, we remark that $\cohd_\ell(K) \leq n+1$ and so $H^r\left(K, \mu_m\right)=0$ for any $r >n+1$. This motivates the term \emph{top cohomology} for $H^{n+1} \left(K, \mu_m\right)$ in this setting. 

\begin{proof}
   We use induction on the number of prime factors of $m$. For $m=\ell$ prime, the result follows from \Cref{lemma:muellnotink}. If $m = s \ell$ for some prime $\ell$, consider the short exact sequence of $K$-modules 
    \begin{equation}\label{split-mun}
    \xymatrix{ 0 \ar[r] & \mu_\ell \ar[r] & \mu_{m} \ar[r]^-{\ell} & \mu_s \ar[r] & 0 }. 
    \end{equation} 
    For any $\xi \in H^{n+1}\left(K, \mu_m\right)$, let $\tilde{\xi}$ be its image under the map  
    $$\xymatrix{ H^{n+1} \left(K, \mu_m\right) \ar[r]& H^{n+1}\left(K, \mu_s\right) }$$
    arising from the long exact sequence in cohomology. By the inductive hypothesis, there is an extension $L_1$ of $K$ of degree at most $s$ such that $L_1$ splits $\tilde{\xi}$. By exactness of the long exact sequence of cohomology induced by (\ref{split-mun}), the restriction $\xi_{L_1}$ of $\xi$ to $L_1$ lifts to an element $\eta \in H^{n+1}\left(L_1, \mu_\ell\right)$.
    $$\xymatrix{
    H^{n+1}\left(L_1, \mu_\ell\right) \ar[r] & H^{n+1} \left(L_1, \mu_m\right) \ar[r]& H^{n+1}\left(L_1, \mu_s\right) \\
    H^{n+1}\left(K, \mu_\ell\right) \ar[r] \ar[u]_{\mathrm{res}}& H^{n+1} \left(K, \mu_m\right) \ar[u]_{\mathrm{res}}\ar[r]& H^{n+1}\left(K, \mu_s\right)\ar[u]_{\mathrm{res}} 
    }$$
    By \Cref{lemma:muellnotink}, there is a field extension $L$ of $L_1$ of degree at most $\ell$ so that $L$ splits $\eta$. As a result, $\xi_L= \eta_L =0$ and so $L$ is a splitting field of $\xi$ of degree 
    $$[L:K] = [L:L_1][L_1:K] \leq \ell s = m. $$ 
    The statement follows. 
\end{proof}

    \section{The main theorems}\label{main}

In this section, we prove our two main theorems. First, we prove an analogue of Tate's common slot lemma over semiglobal fields. When considering coefficients in $\mu_\ell$ for $\ell$ prime, we assume that $\cohd_\ell(k) = n$. For non-prime $m$, there is no notion of $m$-cohomological dimension and therefore we assume instead that $\cohd(k) =n$. Let $F$ be a semiglobal field with residue field $k$. We show in \Cref{commonslot} that any finite set $\mathcal{S}$ of $r$ symbols in $H^{n+2}\left(F, \mu_m^{\otimes n+2}\right)$ for $m$ coprime to the characteristic of $k$ is split by a common extension $L$ of $F$ of degree $m$. We first consider the case that $m= \ell$ is prime and we recall from \Cref{HHK} that 
$$\Sha_{\Omega_L}^{n+2} = \Ker \left( H^{n+2} \left( L, \mu_\ell^{\otimes r} \right) \ra \prod_{\omega \in \Omega_{L}} H^{n+2} \left( L_\omega, \mu_\ell^{\otimes r} \right) \right)=0$$
for any integer $r$. We denote the restriction of any divisorial valuation $ \omega \in \Omega_L$ to $F$ by $\nu$. As $\Sha_{\Omega_L}^{n+2} =0$, it suffices to show that, for any $s \in \mathcal{S}$, the restriction $s_{F_\nu}$ splits over $L_\omega$ for any divisorial valuation $\omega \in \Omega_L$. We then extend the result to the case that $m$ is any number coprime to the characteristic of $k$. 

If $\mu_m \subset k$, then \Cref{commonslot} implies that period and index agree in the top cohomology for semiglobal fields (see \Cref{global1}). In the following, we remove the assumption that $\mu_m \subset k$ while retaining the equality for period and index in top cohomology. More precisely, in \Cref{thm:maintheorem} we show that, for a semiglobal field $F$, any class $\xi$ in $H^{n+2}\left(F, \mu_m\right)$ has index bounded above by $m$. \\

Recall the following setting for semiglobal fields. Let $K$ be a complete discretely valued field and denote by $\O_K$ its ring of valuations. Denote by $C$ a smooth, projective curve over $K$, and  $F=K(C)$ is the function field of $C$. In particular, $F$ is a semiglobal field. Let $k$ be the residue field of $K$. We discussed models and divisorial places in \Cref{semiglobal}, but let us recall: 

\begin{defn}[Model]
A \emph{regular  model} of the curve $C$ is a flat, projective, regular, integral $\mathcal{O}_K$-scheme $\mathfrak{X}$ with $\mathfrak{X}_K = \mathfrak{X} \times_{\mathcal{O}_K} K = C$. We denote the reduced special fiber by $\mathfrak{X}_k$.  
$$\xymatrix{
\mathfrak{X}_k \ar[d] & \ar[l]\mathfrak{X} \ar[r]\ar[d] & C\ar[d] \\
\Spec(k) & \ar[l] \Spec (\mathcal{O}_K) \ar[r] & \Spec(K) 
}$$
\end{defn}
In the course of the proof, we will choose an appropriate model $\mathfrak{X}$ so that the given cohomology class can only ramify along a divisor with normal crossings. Recall the notation $\Omega_F$ from \Cref{semiglobal} for the set of divisorial places on $F$, which is the set of valuations on $F$ corresponding to codimension one points on all (regular, projective) models of $C$. To split the cohomology class, it will suffice to find an extension of $F$, which splits the cohomology class locally everywhere by \Cref{HHK}. 

We begin with the following.

\begin{lemma}
\label{twodim-decomp}Let $A$ be a complete regular local two-dimensional ring  with field of fractions $E$
and residue field $k$. Let $(\pi, \delta)$ be the maximal ideal of $A$.
For  $m \geq 2$ and $1 \leq i \leq m$, let $a_i = u_i\pi^{r_i}\delta^{s_i} \in A$ for some units $u_i \in A$ and
$r_i , s_i \in \Z$.  Then 
$$(a_1, \ldots , a_{m})  = \xi_1 + (\xi_2, \pi) + (\xi_3,\delta) + (\xi_4, \pi, \delta),$$
where $\xi_1 \in H^{m}_{et}(A, \mu_\ell^{\otimes m})$, $\xi_2, \xi_3 \in H^{n-1}_{et}
(A, \mu_\ell^{\otimes m-1})$,
$\xi_4 \in H^{n-2}_{et}(A, \mu_\ell^{\otimes m-2})$.
\end{lemma}

\begin{proof} The proof follows on the same lines as the proof of \Cref{lemma:decompose_xi_over_K}. We remark that the assumption $\mu_\ell \subset k$ in \Cref{lemma:decompose_xi_over_K} is only necessary to write the cohomology class as a sum of symbols. 
\end{proof}

We now prove an analogue of Tate's common slot lemma for number fields over semiglobal fields. We remark that the roots of unity $\mu_m$ may not be contained in $k$ for this result. 

\begin{thm}\label{commonslot}
    Let $F$ be a semiglobal field over a complete discretely value field $K$, with residue field $k$. 
   Let $\ell$ be a prime not equal to char$(k)$. Assume that $\cohd_\ell(k) =n$. 
   Let 
    $$\mathcal{S} = \left\{ \left( a_{i,1},a_{i,2}, \ldots,a_{i,n+2}\right) : 1 \leq i \leq r \right\}$$
    be a set of $r$ symbols in $H^{n+2}\left( F, \mu_\ell^{\otimes n+2} \right)$. Then there is a field extension $L$ of $F$ of degree $\ell$ which splits all symbols in $\mathcal{S}$.
\end{thm}

\begin{proof}
 We denote by $s_i = \left( a_{i,1}, \ldots, a_{i,n+2}\right)$ for $1 \leq i \leq r$ the symbols in $\mathcal{S}$. By \Cref{HHK}, it suffices to prove the existence of a field extension $L$ of $F$ satisfying $[L:F] \leq \ell$ with the property that the restriction $\left( s_i\right)_{L_\omega}$ of $s_i$ to $L_\omega$ is trivial for any divisorial place $\omega \in \Omega_L$. We will show the stronger statement that this is true for any discrete valuation $\omega$ of $L$, not just any divisorial valuation. \\

We now construct the desired splitting field $L$. Let $\mathfrak{X}$ be any model of $F$ so that the function field of $\X$ is $F$. We recall that $\X_k$ denotes the reduced special fiber of $\mathfrak{X}$. Abusing notation, we denote a representative of $a_{i,j}$ in $F^\times$ by $a_{i,j}$ as well. As an element in the function field of $\mathfrak{X}$, $a_{i,j}$ is a rational function on $\mathfrak{X}$. Let $D_{i,j} = \div(a_{i,j})$ be the corresponding divisor on $\X$. After blowing up, we may assume that there are regular curves $C_i$ so that 
$$ \mathfrak{X}_k \cup \left( \bigcup_{i,j} \mathrm{supp}\left(\mathrm{div}(a_{i,j})\right) \right) =  \bigcup_{i} C_i $$
and the $C_i$ have normal crossings. We denote $\mathcal{C}= \{ C_i\}$. We call a point $P$ on $\mathfrak{X}$ a \emph{nodal point of $\mathcal{C}$} if it lies on both $C_i$ and $C_j$ for some $i \neq j$. Choose a finite set of closed points $\mathcal{P} = \{P_j\}$ that includes all nodal points of $\mathcal{C}$ and at least one point on the components of $\X_k$. We note that even though the $C_i$ may not be on the reduced special fiber $\X_k$, all closed points $P_j$ in $\mathcal{P}$ lie on $\X_k$.

Since $\X$ is projective, there is an affine open subscheme $U$ of $\X$ that contains the finite set of closed points $\mathcal{P} = \{ P_j\}$. We denote the affine coordinate ring of $U$ by $A$. Since $\X$ is regular, $A$ is regular as well. Let $R$ be the semilocalization of $A$ at $\mathcal{P}$. Because $A$ is regular, we deduce that $R$ is a semilocal regular domain and hence a unique factorization domain. Therefore each $C_i$ defines a height one prime ideal in $R$, whose generator we denote by $\pi_i$. Finally, let $g = \prod_i \pi_i$ and set $L= F(\sqrt[\ell]{g})$. We aim to show that $L$ is a splitting field for $s_i$ for all $1 \leq i\leq r$. Any discrete valuation $\omega$ on $L$ restricts to a discrete valuation $\nu$ on $F$.  We claim that $L_\omega$ is a splitting field for the restriction $\left( s_i\right)_{F_\nu}$ of $s_i$ to $F_\nu$ for any $1 \leq i \leq r$. Throughout the remainder of the proof, we fix $1 \leq i \leq r$ and we denote $s = s_i$. We remark that by properness of the model $\X$, any discrete valuation $\nu$ is centered on a closed point or on a point of codimension one. We can define, for each point $P \in \mathcal{P}$, a two-dimensional regular local ring $\mathcal{O}_{\mathfrak{X},P}$. Let $m_{\mathfrak{X},P}$ be the maximal ideal of $\mathcal{O}_{\mathfrak{X},P}$, and denote by  $\widehat{\mathcal{O}_{\mathfrak{X},P}}$  the completion of $\mathcal{O}_{\mathfrak{X},P}$ with respect to this maximal ideal.

\begin{itemize}
    \item \textbf{Case 1: Suppose that $\nu$ is centered on a codimension one point $P$ on $\mathfrak{X}$.} There are two possibilities for $P$, either  $P \not\in \mathcal{C}$ or $P \in \mathcal{C}$.
        \begin{itemize}
            \item \textbf{Case 1a: Suppose that $P \neq C_i$ for all $i$. } By definition of the $C_i$'s and our hypothesis, $a_{ij}$ is a unit in $\O_{\X,P}$ for all $i$ and $j$. Thus $s$ is unramified at the completion $\widehat{\O_{\X,P}}$ and so
            \begin{equation}\label{eq:xi_Fnu} s_{F_\nu} \in H^{n+2}_{\text{nr}} \left( F_\nu, \mu_\ell^{\otimes n+2}\right) \cong H^{n+2}_{et} \left(\widehat{\O_{\X,P}}, \mu_\ell^{\otimes n+2}\right) \cong H^{n+2} \left( k_\nu, \mu_\ell^{\otimes n+2}\right),\end{equation} where the last isomorphism holds since $\widehat{\O_{\X,P}}$ is Henselian. We remark that, since $P$ is a codimension one point, the residue field $k_\nu$ of $F_\nu$ has transcendence degree 1 over $k$, and hence has the property that $\cohd_\ell(k_\nu) = n+1$. Using this and (\ref{eq:xi_Fnu}), we have that $s_{F_\nu} = 0$ and therefore its restriction $s_{L_\omega}$ also vanishes. 
            
 \item \textbf{Case 1b: Suppose that $P = C_i$ for some $i$. }  
Since $g$ is a parameter at $C_i$, $F_\nu(\sqrt[\ell]{g}) = L_w$,
  by \Cref{muellink},  $s_{F_\nu}$ are  split by the extension  $L_w$. 
 \end{itemize}
 
    \item \textbf{Case 2: Suppose that $\nu$ is centered on a codimension two point $P$ on $\mathfrak{X}$}. We distinguish three cases (2a) $P$ is a nodal point of $C_i$ and $C_j$ for some $i \neq j$, (2b) $P$ is a point on exactly one of the $C_i$, and (2c) $P$ does not lie on any $C_i$. Denote by $k(P)$ the residue field of both $\mathcal{O}_{\X,P}$ and its completion.
    \begin{itemize}
        \item\textbf{Case 2a: Suppose that $P$ is a nodal point on $\bigcup_i C_i$.} Since $\mathfrak{X}$ has normal crossings, a nodal point will be at the intersection of two distinct curves $C_i$ and $C_j$. Without loss of generality, assume that $P$ is the intersection point of $C_1$ and $C_2$. We remark that the maximal ideal of the local ring $\O_{\mathfrak{X},P}$ is generated by $\pi_1$ and $\pi_2$. Denote the completion of $\O_{\mathfrak{X},P}$ by $\widehat{\O_{\mathfrak{X},P}}$, a complete local ring of dimension $2$. The residue field $k(P)$ is a finite field extension of $k$, thus $\cohd_\ell(k(P)) = \cohd_\ell(k) =n$. Furthermore, the field of fractions of $\widehat{\O_{\mathfrak{X},P}}$ is contained in the field $F_\nu$, which is a complete discretely value field.

   By \Cref{twodim-decomp},  we can decompose $s_{F_\nu}$ as
        $$s_{F_\nu} = \xi_1 + \left( \xi_2 ,\pi_1\right) + \left(\xi_3,\pi_2\right) + \left(\xi_4 ,\pi_1,\pi_2\right),$$
        for unramified cohomology classes  $\xi_1 \in H^{n+2}_{et}\left(\widehat{\O_{\X,P}}, \mu_\ell^{\otimes n+2}\right),$ $ \xi_2, \xi_3 \in H^{n+1}_{et}\left(\widehat{\O_{\X,P}}, \mu_\ell^{\otimes n+1}\right),$ and $ \xi_4 \in H^{n}_{et}\left(\widehat{\O_{\X,P}}, \mu_{\ell}^{\otimes n}\right)$. In conclusion, $\xi_1 , \xi_2,\xi_3$ are unramified at $\O_{\mathfrak{X},P}$ and vanish over its completion as 
        $$H^{n+i}_{et}\left(\widehat{\O_{\mathfrak{X},P}}, \mu_\ell^{\otimes n+i}\right) \cong H^{n+i}\left(k(P), \mu_\ell^{\otimes n+i}\right)$$  for $i = 1, 2$ and $\cohd_\ell k(P) =n.$ We conclude that $s_{F_\nu} = \left( \xi_4, \pi_1, \pi_2\right)$. 
        In $\O_{\X,P}$, we have that $g= u \pi_1 \pi_2$ and so
        \begin{align*} s_{F_\nu} &= \left( \xi_4 ,\pi_1, \pi_2\right) \\
        &=  \left(\xi_4,- \pi_1,\pi_1 \pi_2\right) \\
        &= \left( \xi_4 ,-\pi_1,u^{-1} g\right)\\
        &= \left(\xi_4 ,u,-\pi_1\right) + \left( \xi_4 ,-\pi_1,g\right)\\
        &= \left( \xi_4, -\pi_1,g\right), \end{align*} where the last equality holds since $\left( \xi_4,u\right)$ is unramified in 
        $$H^{n+1}\left(\widehat{\O_{\X,P}}, \mu_{\ell}^{\otimes n+1}\right) \cong H^{n+1} \left(k(P), \mu_\ell^{\otimes n+1}\right) = 0,$$
         since $\kappa(P)$  has cohomological dimension $n$. Using \Cref{rmk:splitting-field}, we see that $L_\omega = F_\nu \left( \sqrt[\ell]{g}\right)$ splits $s_{F_\nu}$. 
        \item \textbf{Case 2b: Suppose that $P$ is a point on exactly one of the $C_i$.} Without loss of generality, we may assume that $P$ lies on $C_1$. By \Cref{twodim-decomp}, the class $s_{\widehat{\O_{\X,P}}}$ becomes $\xi_1+(\xi_2,\pi_1)$ for some  $\xi_1 \in
         H^{n+2}_{et}\left( \widehat{\O_{\X,P}}, \mu_\ell^{\otimes n+2}\right)$ and $\xi_2 \in 
         H^{n+1}_{et}\left( \widehat{\O_{\X,P}}, \mu_\ell^{\otimes n+1}\right)$. As in Case 2a, $s$ is split by the degree $\ell$ extension given by attaching an $\ell$-th root of $g = u \pi_1$, where $\pi_1$ is the local uniformizer at $C_1$ for $\nu$.         
        \item \textbf{Case 2c: Suppose that $P$ does not lie on $C_i$ for any $i$.} Then $a_{ij}$ is a unit in $\O_{\X,P}$ and therefore $s_{F_\nu}$ is an unramified cohomology class in 
        $$H^{n+2}_{et}\left( \widehat{\O_{\X,P}}, \mu_\ell^{\otimes n+2}\right) \cong H^{n+2} \left( k(P), \mu_\ell^{\otimes n+2}\right).$$ The residue field $k(P)$ is a finite extension of $k$, implying that $\cohd_\ell(k(P)) = \cohd_\ell(k) = n$. Hence, the class $s_{\widehat{\O_{\mathfrak{X},P}}}$ vanishes, and so $s_{F_\nu}$ does as well. 
    \end{itemize}
\end{itemize}
\end{proof}

\begin{cor}\label{commonslot-m}
   Let $F$ be a semiglobal field over a complete discretely value field $K$, with residue field $k$. 
  Let $m \geq 2$ be  coprime to char$(k)$. Assume that $\cohd(k) =n$. 
  Suppose that  $\mu_m \subset  F$.     Let 
    $$\mathcal{S} = \left\{ \left( a_{i,1},a_{i,2}, \ldots,a_{i,n+2}\right) : 1 \leq i \leq r \right\}$$
    be a set of $r$ symbols in $H^{n+2}\left( F, \mu_m^{\otimes n+2} \right)$. Then there is a field extension $L$ of $F$ of degree $\ell$ which splits all symbols in $\mathcal{S}$.
\end{cor}

\begin{proof} We note that the assumption $\cohd(k) = n$ implies that $\cohd_\ell (k) \leq \cohd(k) =n$ for any prime  $\ell$  dividing $m$.
The proof follows from \Cref{commonslot} and arguments analogous to the proof of \Cref{mun}. 
\end{proof}

If $k$ contains $\mu_m$, the previous theorem implies that period and index agree in top cohomology. We later remove this assumption. 

\begin{prop} \label{global1}
Let $F$ be the function field of the smooth projective curve $C$ over the complete discretely value field $K$, with residue field $k$.  Assume additionally that $\mu_m\subset k$, $m$ is coprime to the characteristic of $k$, and $\cohd(k) =n$.  Then any element  $\xi$  in $H^{n+2}(F, \mu_m)$ has index bounded above by $m$. In fact, there is a field extension $L$ of $F$ of degree $m$ so that $\xi_L =0$.
\end{prop}

\begin{proof}
   Suppose $m=\ell$ is a prime number coprime to the characteristic of $k$ and let $\xi \in H^{n+2}(F, \mu_\ell)$. By Voevodsky's result (see \Cref{Voevodsky}), we may write 
    $$\xi = \sum_{i=1}^r \left( a_{i,1}, \ldots, a_{i,n+1}\right) $$
    for some $a_{i,j} \in F^\times$. We denote  $$\mathcal{S} = \left\{ \left( a_{i,1},a_{i,2}, \ldots,a_{i,n+2}\right) : 1 \leq i \leq r \right\}.$$ By \Cref{commonslot} there is some field extension $L$ of $F$ of degree $\ell$ which splits all symbols in $\mathcal{S}$. As a result, the sum $\xi$ of the symbols in $\mathcal{S}$ also splits over $L$. The general case $m$ coprime to the characteristic of $k$ with $\mu_m \subset F$ follows arguments by induction analogous to the proof of \Cref{mun}.
\end{proof}

From now on, we no longer require that $\mu_m \subset k$. We first assume that $m=\ell$ is a prime number. Consider the field extension $ F[\rho]$ obtained by attaching a primitive $\ell$-th root of unity $\rho$. We denote by $G$ the Galois group of $F[\rho]$ over $F$. To descend \Cref{maintheoremprime} to this setting, we first need to construct an appropriate model $\mathfrak{X}$ for a fixed cohomology class $ \xi \in H^{n+2}\left(F, \mu_\ell\right).$
Consider the restriction $\xi_{F[\rho]} \in H^{n+2}\left( F[\rho], \mu_\ell\right)$. Then $\xi_{F[\rho]}$ may be written as the sum of symbols 
$$\xi_{F[\rho]} = \sum_{i=1}^r \left( a_{i,1}, \ldots, a_{i,n+2}\right) $$
for some $a_{i,j} \in F[\rho]$. 

\begin{lemma} 
With the notation as above, there is a regular proper model $\X$ over $\O_K$ of $F$ such that the following conditions are satisfied: 
\begin{enumerate}
    \item The base change $\X[\rho]$ of $\X$ to $\O_{K[\rho]}$ is a regular proper model of $F[\rho]$;
    \item There are distinct irreducible curves $C_1, \ldots, C_r$ on $\X[\rho]$ so that 
            \begin{equation}\label{eq:coverXrho}\left( \X[\rho]\right)_{k[\rho]} \cup \left( \bigcup_{i,j} \mathrm{supp}\left(\mathrm{div}(a_{i,j})\right) \right)  \subseteq \bigcup_{\substack{1 \leq i \leq r\\ \sigma \in G }} C_i^\sigma,\end{equation}
            where $C_i^\sigma$ denotes the Galois conjugate of $C_i$; 
    \item For the projection $\pi: \X[\rho] \ra \X$, we have that $\pi(C_i) = D_i$ for $1 \leq i \leq r$ are distinct irreducible curves on $\X$; and  
    \item $\sum_{i,\sigma} C_i^\sigma$ on $\X[\rho]$ and $\sum_{i} D_i$ on $\X$ are divisors with normal crossings. 
\end{enumerate}
\end{lemma} 

\begin{proof}
Let $\X$ be a regular proper model of $F$ over $\O_K$ as discussed in \Cref{semiglobal}  and denote by $\X[\rho]$ its base change to $\O_{K[\rho]}$. We remark that the projection $ \pi: \X[\rho] \ra \X$
is a finite \'etale morphism. We may choose distinct irreducible curves $D_1, \ldots, D_r$ on $\X$ so that if $C_1, \ldots, C_r$ are distinct irreducible curves on $\X[\rho]$ with $\pi(C_i) = D_i$, then (b) is satisfied. \\
We remark that the model $\X$ does not necessarily satisfy (d).  
 Let $b: \tilde{\X} \ra \X$ be a blow up such that $b^*\left( \sum_{i=1}^r D_i \right)$ is a normal crossing divisor on $\tilde{\X}$. We denote by $\tilde{\X}[\rho]$ the base change, by $\tilde{\pi}: \tilde{\X}[\rho] \ra \tilde{\X}$ the corresponding \'etale projection, and by $b[\rho]: \tilde{\X}[\rho] \ra \X[\rho]$ the base change of $b$. Then the following diagram commutes 
 $$\xymatrix{
\tilde{\X}[\rho]\ar[d]_{\tilde{\pi}} \ar[r]^{b[\rho]} & \X[\rho] \ar[d]^\pi\\
\tilde{\X} \ar[r]^b & \X  
 }$$
and therefore 
$$ \tilde{\pi}^* b^* \left( \sum_{i=1}^r D_i \right) = \left( b[\rho]\right)^* \pi^* \left( \sum_{i=1}^r D_i \right) = \left( b[\rho]\right)^* \left( \sum_{\substack{1 \leq i \leq r \\ \sigma \in G}} C_i^\sigma \right). $$
Since $b^* \left( \sum_{i=1}^r D_i \right)$ is a normal crossing divisor on $\tilde{\X}$ and $\tilde{\pi}: \tilde{\X}[\rho] \ra \tilde{\X}$ is finite \'etale, the pullback $\left( b[\rho]\right)^* \left( \sum_{i, \sigma} C_i^\sigma \right)$ is a divisor with normal crossings on $\tilde{\X}[\rho]$. This shows that the model $\tilde{\X}$ satisfies the conditions (a)-(d) in the statement. 
\end{proof} 

We are now ready to split the cohomology class $\xi \in H^{n+2}\left( F, \mu_\ell\right)$ using the model constructed in the previous lemma. 

\begin{thm} \label{maintheoremprime}
 Let $F$ be a semiglobal field over a complete discrete-valued field $K$ with residue field $k$. Let $\ell$ be a prime number coprime to the characteristic of the residue field $k$, and suppose that $\cohd_\ell(k)=n$. Any element  $\xi$  in $H^{n+2}\left(F, \mu_\ell\right)$ is split in a field extension $L$ of $F$ of degree $\leq \ell$ so that $\ind(\xi)$ divides $\ell$.
\end{thm}

\begin{proof}
In view of \Cref{global1}, we may assume that $F$ does not contain a primitive $\ell$-th root of unity. Let $\rho$ be a primitive $\ell$-th root of unity in $\bar{F}$ and remark that $\left[F[\rho]: F\right]=\ell -1$. 
We write  $\xi_{F[\rho]} = \sum_{1\leq i \leq s}\left(a_{i,1}, \ldots ,a_{i,n+2}\right),$ with $a_{i,j} \in F[\rho]$. Let $\X$ be a regular proper model chosen as in the previous Lemma, i.e. satisfying conditions (a)-(d). We denote $\mathcal{C} = \{ C_1, \ldots, C_r\}$ and $\mathcal{C}^\sigma = \left\{ C_i^\sigma : 1 \leq i \leq r, \sigma \in G\right\}$, where $G$ is the Galois group of $F[\rho]$ over $F$.

As in the proof of \Cref{global1}, we take $g \in F[\rho]$ so that 
\begin{equation} \label{eq:divg}
\div_{\X[\rho]}(g) =\sum_{i,\sigma}C_i^{\sigma} +E,\end{equation} where $E$ is a divisor on $\X[\rho]$ whose support does not contain any $C_i^{\sigma}$ and no component of $\Supp(E)$ passes through a nodal point of $\mathcal{C}^\sigma$.
Let $N: F[\rho] \rightarrow F$ denote the norm map and we recall that $\pi: \X[\rho] \ra \X$ denotes the projection.
Let us also denote $k(C)$ (resp. $k(D)$) the residue field at the codimension one point $C$ on the 2-dimensional scheme $\X[\rho]$ (resp. $\X$). 
Then
\begin{equation}\label{eq:divNg} \begin{aligned}\div_{\X[\rho]}\left(N(g)\right) &= \pi^*\pi_*\left( \div_{\X[\rho]}(g)\right)\\
&\overset{(\ref{eq:divg})}= \pi^*\pi_*\left(\sum_{i,\sigma}C_i^{\sigma}+E \right)\\
&= \pi^*\left(\sum_i f_ig_iD_i\right)+ \pi^*\pi_*E\\
&=(\ell -1)\sum_{i,\sigma} C_i^\sigma + \pi^*\pi_*E,\end{aligned}\end{equation}
where $f_i$ is the degree of the residue field extension $\left[k(C_i): k(D_i) \right]$ and $g_i$ is the number of conjugates $C_i^{\sigma}$ of $C_i$. For the last equality, we remark that $$f_ig_i= [F[\rho]: F]= (\ell -1).$$

We now prove that no component of $\pi^*\pi_*E$ passes through a nodal point of $\mathcal{C}^\sigma$. Each component of $\pi^*\pi_*E$ is of the form $E^\tau_r$, with $E_r$ a component of $\Supp(E)$ and $\tau \in G$. Suppose that $E^\tau_r$ passes through a nodal point $C_i^{\sigma_1} \cap C_j^{\sigma_2}$ of $\mathcal{C}^\sigma$. Then $E_r$ passes through $C_i^{\sigma_1\tau} \cap C_j^{\sigma_2 \tau}$ which is a nodal point of $\mathcal{C}^\sigma$ contradicting that no component of $E$ passes through a nodal point of $\mathcal{C}^\sigma$.\\

As in the proof of \Cref{commonslot}, let $\mathcal{P}$ be a set of closed points on $\X[\rho]$ that includes all nodal points of $\mathcal{C}$ and at least one point on each component. We denote the set $\mathcal{P}^\sigma = \left\{ P^\sigma : P \in \mathcal{P}, \sigma \in G\right\}$. Let $U$ be an open subscheme of $\X[\rho]$ that contains $\mathcal{P}^\sigma$ and denote its affine coordinate ring by $A$. In the semilocalization $R$ of $A$ that corresponds to $U - \mathcal{P}^\sigma$ each $C_i^\sigma$ defines defines a height one prime ideal whose generator we denote by $\pi_i^\sigma$. Without loss of generality, we may assume that $\sigma(\pi_i) = \pi_i^\sigma.$ We denote $L = F\left( \sqrt[\ell]{N(g)}\right)$ and we claim that $L[\rho]$ splits $\xi_{F[\rho]}.$ Let $\omega$ be a discrete valuation of $L[\rho]$ and denote its restriction to $F[\rho]$ by $\nu$. If $\nu$ is centered on a codimension one point $P$ on $\X[\rho]$ so that $P \not\in \mathcal{C}$, then the argument in the proof of \Cref{commonslot} Case 1a applies to show that $\xi_{F[\rho]_\nu}$ splits over $\xi_{L[\rho]_\omega}$. If $P$ is centered on $C_i^\sigma$ for some $i$ and some $\sigma$, then we remark that $N(g) = u \left(\pi_i^\sigma\right)^{f_i}$ for some unit $u \in \O_{\X[\rho],C_i^\sigma}$. Since $f_i$ is coprime to $\ell$, the $\ell$-torsion symbol 
$\left(\xi_2, \pi_i\right)$ for $\xi_2 \in H^{n+1}\left( \widehat{\O_{\X,C_i^\sigma}}, \mu_\ell\right)$ is split if and only if the symbol $\left(\xi_2, (\pi_i^\sigma)^{f_i}\right) = \left( \xi_2, \pi_i^\sigma\right)^{f_i}$ is split. Thus the argument in the proof of \Cref{commonslot} Case 1b implies that  $\xi_{F[\rho]_\nu}$ splits over $L[\rho]_\omega$. A similar reasoning shows the local splitting if $\nu$ is centered on a codimension-two point on $\X[\rho]$. We conclude that $L[\rho]$ splits $\xi_{F[\rho}$. It remains to show that $\xi$ is already split over $L$. \\

We consider the following commutative diagram 
$$
\xymatrix{ 
H^{n+2}\left(L, \mu_{\ell}\right) \ar@{^(->}[r]^{\res} & H^{n+2} \left(L[\rho], \mu_{\ell}\right) \\
H^{n+2}\left(F, \mu_{\ell}\right) \ar@{^(->}[r]^{\res} \ar[u]^\res& H^{n+2} \left(F[\rho], \mu_{\ell}\right) \ar[u]^\res
}.
$$
The horizontal restriction maps are injective since $[F[\rho]:F] = [L[\rho]:L] \ell-1$ is coprime to $\ell$. Since $\xi$ splits over $L[\rho]$, it also splits over the degree $\ell$ extension $L$ of $F$ by injectivity. 
\end{proof}

We note that we have therefore, in effect, proved the following theorem.

\begin{thm}\label{thm:maintheorem}
    Let $F$ be a semiglobal field over $K$ a discrete valued field and denote $k$ the residue field of the latter. Let $m$ be any positive integer. Suppose that the characteristic of $k$ is coprime to $m$ and $\cohd(k) =n$. Then any element $\xi \in H^{n+2}\left(F, \mu_m\right)$ has index bounded above by $m$. 
\end{thm}

\begin{proof}
Same argumentation as in \Cref{mun}, using \Cref{maintheoremprime} for the inductive assumption. 
\end{proof}

\section{Open questions}\label{sec:open}

In this section, we discuss related results and open questions. We first propose a conjecture which extends our results on bounded index for the top cohomology for semiglobal fields. Secondly, we highlight the relationship between symbol length and the $u$-invariant for quadratic forms. 

\subsection{Questions on top cohomology} \label{sec:questions-top-cohom}

Let $F$ be any field and fix a prime number $\ell$ that is coprime to the characteristic of $F$. We assume that $\cohd_\ell(F)=n$, so that $H^{n+r}(F, \mu_\ell) =0$ for any $r>0$. We call $H^n(F, \mu_\ell)$ the top cohomology of $F$. \\
 We remark that in this generality, there is no uniform bound for indices of elements in top cohomology $H^n(F, \mu_\ell)$. More precisely, in \cite{Merkurjev1991}, Merkurjev produces examples of fields $F$ with $\cohd(F) =2$ which admit division algebras that are tensor products of $n$ quaternion algebras for any given $n$. These algebras have index $2^n$ in $H^2(F, \mu_2)$. We are interested in uniform bounds for indices in top cohomology. As a first step in this program, we highlight the following conjecture. 

\begin{conjecture} \label{conj}
    Let $K$ be a field with $\cohd_\ell(K) = n$ and fix a prime $\ell$ coprime to the characteristic of $K$. Let $F$ be a function field of a curve over $K$. If the indices in $H^n (K, \mu_\ell)$ are bounded, then the indices in $H^{n+1}(F, \mu_\ell)$ are bounded. 
\end{conjecture}

We prove the conjecture if $F$ is a semiglobal field. The conjecture also has a positive answer if $K$ is a global field of characteristic different from $\ell$ \cite{Sur20}. 

We recall that the diophantine dimension $\dd(F)$ of a field $F$ is the smallest integer $r \geq 0$ such that any $n$-form of degree $d\geq 1$ has a nontrivial zero in $F$, whenever $n>d^r$. If no such $r$ exists, we say that $\dd(F)=\infty$. The following conjecture was posed by Krashen in \cite[Conjecture 1, p. 997]{krashen2016}.

\begin{conjecture}
 Suppose that $\dd(F) \leq d$. Then for elements $\xi \in H^d\left( F, \mu_\ell^{\otimes d}\right)$, the index $\ind(\xi)$ divides $\ell$. 
\end{conjecture}

Our main result may be viewed as an analogue of this conjecture with the assumption $\dd(F)\leq d$ replaced by $\cohd(k) \leq d-2$, which implies $\cohd(F) \leq d$, for a semiglobal field $F$.

\subsection{Quadratic forms and $u$-invariant} \label{quad}

It is well-known that classical invariants of quadratic forms take values in the Galois cohomology groups. In 1970, Milnor proposed higher invariants for quadratic forms which could determine the isomorphism class of a quadratic form up to hyperbolic planes. 

Let $W(K)$ be the Witt ring of a field $K$ and denote by $I(K)$ the fundamental ideal of $W(K)$ of even dimensional forms. The ideal $I^n(K)$ is generate by $n$-fold Pfister forms
$$\left< \left< a_1, \ldots, a_n \right> \right> = \left< 1, -a_1\right> \otimes \left< 1, -a_2\right> \otimes \cdots \otimes \left< 1, -a_n\right>.$$
For every $n>0$, Milnor defined a map
$$I^n(K) \rightarrow H^n\left( K, \mu_2\right)$$
that takes the Pfister form $\left< \left< a_1, \ldots, a_n \right> \right>$ to the symbol $(a_1, \ldots, a_n)$. Milnor conjectured that these maps are well-defined and onto with kernel precisely $I^{n+1}(K)$. Milnor's conjecture was proved by Voevodsky in \cite{Voevodsky2003}. 

Another invariant of fields associated with quadratic forms is the $u$-invariant which is defined as the maximum dimension of an anisotropic quadratic form over $K$. In the case that $K$ is a complete discretely valued field with residue field $k$, Springer shows in \cite{Springer1955} that $u(K) = 2u(k)$. Here, we focus on the case of a semiglobal field $F$, which is the function field of a curve $C$ over $K$. Suppose that $\hbox{char}(k) \neq 2$. If $u(l) \leq d$ for any finite extension $l$ of the residue field $k$ and $u(k(C)) \leq 2d $, where $C$ is any curve over $k$, then $u(F) \leq 4d$ by \cite{HHK1}. The case that $K$ is a $p$-adic field was proved in \cite{Parimala-Suresh2010}. On the other hand, if $k$ is perfect with $\mathrm{char}(k) = 2$, then Parimala and Suresh show in \cite{parimalasuresh} that if $F$ is the function field of a curve over a complete discretely valued field $K$ of characteristic zero, then $u(F) \leq 8$. Further, they conjecture that with the exact same hypothesis, but that of $k$ being perfect, $u(F) \leq 8[k: k^2]$. Equivalently, the question is whether every quadratic form in at least nine variables over the function field of a curve over a local field has a nontrivial zero.

We now recall two main results in the area. The first of Krashen, \cite[4.2]{krashen2016}, relates the bounds on indices 
of Galois cohomology classes for any field to bounds on the symbol length in Galois cohomology. The second result relates the bound on symbol length in $\mu_2$-cohomology  to bounds on the $u$-invariant. Recall the notation from \Cref{Voevodsky} on the symbol length of an element in $H^n(F, \mu_m)$ as $\hbox{symb}_n^m(F)$. 

\begin{thm}[{\cite{krashen2016}}]\label{thm:krashen2016}
Let $K$ be a field, $\ell$ a prime not equal to the characteristic of $K$ and $n \geq 1$. Suppose that there exists an integer $N$ such that for every finite extension $L$ of $K$ and for all $\beta \in H^d(L, \mu_{\ell}^{\otimes d}), ~~ 1 \leq d \leq n, \ind(\beta) \leq N$. Then, for any $\alpha \in H^n(K, \mu_{\ell}^{\otimes n})$, the symbol length of $\alpha$ is bounded in terms of $\ind(\alpha)$, $N$ and $n$.
\end{thm}

The following theorem is a consequence of a theorem of Orlov, Vishik, and Voevodsky \cite{orlov2007v} on the Milnor conjecture (see also \cite{PS2015}).

\begin{thm}[\cite{orlov2007v}] \label{mu2u}
Let $K$ be a field of characteristic not equal to 2. Suppose that there exist integers $M \geq 1$ and $N$ such that $H^M(K, \mu_2) = 0$ and $\hbox{symb}_d^2(K)(\alpha) \leq N$ for all $\alpha \in H^d(K, \mu_2), ~~ 1\leq d < M$. Then the $u$-invariant is bounded by a function of $M$ and $N$.
\end{thm}

If we could show that the index is bounded for elements in $H^i (F, \mu_2)$ for all $i \leq n+2$, then using \Cref{thm:krashen2016,mu2u}, it would follow that the $u$-invariant of $F$ is bounded. 
In this paper, letting $\cohd(K)= n+1$ for a discrete valued field $K$, we have shown that in the top cohomology (i.e degree $n+2)$ and in the context of a semiglobal fields $F$ over $K$, the index is necessarily bounded. If we were to show that the index is bounded for all $i$-cohomology, $i \leq \cohd(K)$, then we would obtain a bound depending on this index for each $\alpha \in  H^{n+1}(F, \mu_{\ell}^{\otimes n+1})$. Then applying \Cref{mu2u}, with $M=n+2$, we would obtain a bound on the $u$-invariant as a function of three variables: the integer $n+2$, some homogeneous bound on the indices for all finite extensions of $F$ (in all degrees) and the index of the element in $H^{n+1}(F, \mu_{\ell}^{\otimes n+1})$. 

Finally, we remark that the converse also holds: The finiteness of the $u$-invariant for $\mu_2$-cohomology implies the boundedness of the index in all degrees, as proven in \cite{Saltman}. Further, if the u-invariant of $F$ is bounded by $2^n$ where $\cohd_2(F)=n$, then by \cite[Proposition 5.2(ii)]{krashen2016}, the index of elements in $H^n(F, \mu_2)$ is at most $2$.

\bibliography{bibl}
\bibliographystyle{alpha}

\end{document}